\documentclass[11pt,a4paper,twoside]{amsart}
\usepackage{amsmath}
\usepackage{amssymb}
\usepackage{amsthm}
\usepackage{mathrsfs}

\allowdisplaybreaks[4]

\usepackage{amsrefs}
\usepackage{hyperref} 

\hypersetup{
    colorlinks=true, 
    linktoc= page,   
}

\numberwithin{equation}{section}

\newcommand{\wha}[1]{\widehat{#1}}

\newcommand{\Lebn}[2]{\left\lVert #1 \right\rVert_{L^{#2}}}

\newcommand{\C}{\mathbb{C}}
\newcommand{\R}{\mathbb{R}}
\newcommand{\Z}{\mathbb{Z}}

\renewcommand{\Im}{\operatorname{Im}}
\renewcommand{\Re}{\operatorname{Re}}

\def\({\left(}
\def\){\right)}
\def\<{\left\langle}
\def\>{\right\rangle}
\def\le{\leqslant}
\def\ge{\geqslant}

\def \wt{\widetilde}

\def \e{\varepsilon}

\def \D{\Delta}

\def \pa{\partial}

\def \s{\sigma}
\def \a{\alpha}

\def \t{\theta}

\def \ta{\tau}

\newcommand{\eps}{\varepsilon}

\theoremstyle{plain}
\newtheorem{theorem}{Theorem}[section]
\newtheorem{definition}[theorem]{Definition}
\newtheorem{lemma}[theorem]{Lemma}

\theoremstyle{remark}
\newtheorem{remark}[theorem]{Remark}

\begin{document}
\title[Lifespan of solutions to NLS with homogeneous nonlinearity]{Lifespan of solutions to nonlinear Schr\"odinger equations with general homogeneous nonlinearity of the critical order}
\author[H. Miyazaki]{Hayato Miyazaki}
\address[]{Advanced Science Program, Department of Integrated Science and Technology, National Institute of Technology, Tsuyama College, Tsuyama, Okayama, 708-8509, Japan}
\email{miyazaki@tsuyama.kosen-ac.jp}
\author[M. Sobajima]{Motohiro Sobajima}
\address[]{Department of Mathematics, Faculty of Science and Technology, Tokyo University of Science, 2641
Yamazaki, Noda-shi, Chiba, 278-8510, Japan}
\email{msobajima1984@gmail.com}
\date{}

\maketitle

\begin{abstract}
This paper is concerned with the upper bound of the lifespan of solutions to nonlinear Schr\"odinger equations with general homogeneous nonlinearity of the critical order. In \cite{MM2}, Masaki and the first author obtain the upper bound of the lifespan of solutions to our equation via a test function method introduced by \cite{QZ1, QZ2}. Their nonlinearity contains a non-oscillating term $|u|^{1+2/d}$ which causes difficultly for constructing an even small data global solution. The non-oscillating term corresponds to the $L^1$-scaling critical. 
In this paper, it turns out that the upper bound can be refined by employing a unified test function by Ikeda and the second author \cite{IS1}.
\end{abstract}

\section{Introduction}
In this paper, we consider the lifespan of solutions to nonlinear Schr\"odinger equations 
\begin{align} 
	i \pa_t u + \Delta u = F(u),\; (t,x) \in \R^{1+d} \label{nls} \tag{NLS}
\end{align}
where $u=u(t,x)$ is a complex-valued unknown function, and the nonlinearity $F$ is homogeneous of degree $1+2/d$, that is, $F$ satisfies the condition
\begin{align}
	F(\lambda z) = \lambda^{1+\frac2d} F(z) \label{eq:cond1}
\end{align}
for any $z \in \C$ and $\lambda >0$. It is well-known that the exponent $1+2/d$ is critical in view of large time behavior of solutions to \eqref{nls}. 
Indeed, considering $F(u) = \eta |u|^{p-1}u$ with $p>1$ and $\eta \in \R \setminus \{0\}$, the solution asymptotically behaves like a free solution as $t \to \infty$, when $p>1+2/d$.  On the other hand, if $p=1+2/d$, then the solution behaves like a free solution with a logarithmic phase correction
\begin{align}
	(2it)^{-\frac{d}{2}}  e^{i \frac{|x|^2}{4t}} \wha{u_{+}}\(\frac{x}{2t}\) \exp \( -\eta i \left| \wha{u_{+}}\(\frac{x}{2t}\) \right|^{\frac{2}{d}} \log t \) \label{eq:ma}
\end{align}
for suitable function ${u_+}$ as $t \to \infty$, where $\wha{u_+}$ denotes the Fourier transform of $u_+$ with respect to $x$ (cf. \cite{TY, O1, GO, HN}).
Furthermore, in the critical case,
the behavior of the solutions depends on the shape of the nonlinearity. More precisely, under \eqref{eq:cond1} and some summability assumption on $\{g_n\}_n$, Masaki and the first author \cite{MM1} introduce a decomposition of the nonlinearity
\begin{align}
	F(u) = g_0 |u|^{1+\frac2d} + g_1 |u|^{\frac2d}u + \sum_{n\neq 0,1} g_n |u|^{1+\frac2d-n}u^n \label{eq:decomp}
\end{align}
with the coefficients
\begin{align*}
	g_n = \frac1{2\pi} \int_0^{2\pi} F(e^{i\theta}) e^{-in\theta} d\theta. 
\end{align*}
By using the decomposition, they prove that if $g_0=0$ and $g_1 \in \R$, then \eqref{nls} admits a solution 
which asymptotically behaves like \eqref{eq:ma} with $\eta = g_1$ (see also Masaki, the first author and Uriya \cite{MMU}, and references therein).


Later on, in \cite{MM2}, they try to deal with the large time behavior of the solution to \eqref{nls} under the case $g_0 \neq 0$ in which the nonlinearity contains the non-oscillating term $|u|^{1+2/d}$. However, according to the previous works such as \cite{II, IW, S, ST}, we do not always have even small data global solutions to \eqref{nls} in this case. Hence it seems difficult to find certain large time behavior of the solution. For this reason, instead of specifying the large time behavior, they denies the existence of the solution which behaves like a free solution with or without a logarithmic phase correction like \eqref{eq:ma} with $\eta =g_1$. They further obtain the upper bound of the lifespan of the solution with a general initial data corresponding to a finite time blowup result for small data (see Theorem \ref{thm:1}). 
In this paper, we aim to refine the upper bound of the lifespan given by \cite{MM2}.
Here, let us briefly review previous analysis of blowup phenomena and the lifespan of solutions to the nonlinear Schr\"odinger equation with non-oscillating nonlinearity 
\begin{align}
	i \pa_t u + \Delta u = \rho |u|^p, \quad \rho \in \C \setminus \{0\},\quad p>1. \label{ngnls}
\end{align}
Ikeda and Wakasugi \cite{IW} discovered the blowup phenomena of solutions to \eqref{ngnls} when $p \le 1+2/d$. In \cite{II}, Ikeda and Inui obtained the upper bound of the lifespan of solutions to \eqref{ngnls} under slowly decaying initial data if $p < 1+2/d$. Further, 
Fujiwara and Ozawa \cite{FO} give an estimate of the lifespan of solutions to \eqref{ngnls} in $d$-dimensional Torus.
Also, Oh, Okamoto and Pocovnicu \cite{OOP} investigate stability of the finite time blowup solution constructed
in \cite{II} under a stochastic setting. In the $L^1$-scaling critical case $p=1+2/d$, Ikeda and the second author \cite{IS1} give an upper bound of the lifespan of solutions to \eqref{ngnls} which can be regarded as a refinement of \cite{IW} and \cite{FO}.
They introduce a unified test function, in order to treat the lifespan estimate of blowup solutions to semilinear evolution equations in $L^1$-scaling critical setting such as \eqref{ngnls}. Their test function method is an improvement of that given by Mitidieri and Pokhozhaev \cite{MP}. 

\section{Main results}

Before stating the main results, we recall the technique for the decomposition of the nonlinearity like \eqref{eq:decomp} developed by \cite{MM1}.
We identify a homogeneous nonlinearity $F$ and $2\pi$-periodic function $g$ as follows:
A homogeneous nonlinearity $F$ is written as
\[
	F(u) = |u|^{1+\frac{2}d} F\( \frac{u}{|u|} \).
\]
Let us here introduce a $2\pi$-periodic function $g(\theta)=g_F(\theta)$ by $g_F(\theta) = F(e^{i\theta})$.
Conversely, for a given $2\pi$-periodic function $g$, we can construct a homogeneous nonlinearity $F=F_g:\C \to \C$ by
$F_g(u) = |u|^{1+\frac{2}d} g\( \arg u \)$ if $u\neq 0$ and
$F_g(u) = 0$ if $u=0$.
Since $g(\theta)$ is $2\pi$-periodic function, it holds that 
\[
	g(\theta)=\sum_{n\in\Z} {g}_n e^{in\theta}, \quad g_n := \frac{1}{2\pi} \int_0^{2\pi} g(\t) e^{-in\t}\ d\t
\]
under at least $\{g_n\} \in \ell^1(\Z)$. Hence, the above expansion gives us \eqref{eq:decomp}.

In this paper, we refine the upper bound of the lifespan in \cite{MM2} by employing the technique in \cite{IS1}.
To state the main results, let us give the definition of a weak solution and the lifespan of the solution.

%

%

\begin{definition}[weak solution]
Suppose that $F(z)$ is locally uniformly bounded. 
We say that a function $u(t,x) \in \mathscr{S}'((-\infty,T)\times \R^d)$
is a weak solution to \eqref{nls} with initial data $u(0,x) = u_0(x) \in L^1_{\mathrm{loc}}(\R^d)$ on $[0,T)$,
$T>0$, if $u \in L^{{(d+2)}/{d}}_{\mathrm{loc}} ( (0, T) \times \R^d)$
 and the identity
\begin{align}
\begin{aligned}
	&{}\int_{(0,T) \times \R^d} u(t,x)(-i\pa_t \psi(t,x) + \Delta\psi(t,x)) dxdt\\
	={}& i\int_{\R^d} u_0(x) \psi(0,x)  dx + \int_{(0,T) \times \R^d} F(u(t,x))\psi(t,x) dxdt
\end{aligned}
	\label{weak:1}
\end{align}
holds for any test function $\psi \in C_0^{\infty} ((-\infty,T) \times \R^d)$.
\end{definition}
\begin{definition}[Lifespan of solutions]
For a given data $u_0\in L^1_{\mathrm{loc}}(\R^d)$, we define the maximal existence time by
\begin{align*}
 	T_{\max} =
	T_{\max}(u_0)
	 := \sup \left\{T >0\; ;\;
	\begin{aligned}
	&\text{There exists a weak solution $u(t)$}\\
	&\text{to \eqref{nls} with $u(0) = u_0$ on $[0, T)$}
	\end{aligned}
	\right\}.
\end{align*}
\end{definition}

By the test function method introduced by Zhang \cite{QZ1, QZ2}, the first author and Masaki \cite{MM2} give an upper bound of the lifespan under $g_0 =1$. Remark that without loss of generality, we may let $g_0=1$ by change of variable.

\begin{theorem}[\cite{MM2}] \label{thm:1}
Let $d\ge1$ and $\e >0$. 
Suppose that $\{g_n\}_n \in \ell^1(\Z)$ satisfies $g_0=1$ and
$\mu := g_0- \sum_{n\neq0} |g_n|>0$.
If $f \in L^{1}_{\mathrm{loc}}(\R^d)$ satisfies
\begin{align*}
	-\Im f(x) \ge 
	\left\{
	\begin{aligned}
	&|x|^{-k} && |x| >R_0, \\
	&0 && |x| \le R_0,
	\end{aligned}
	\right.
\end{align*}
for some $k\le d$ and $R_0>0$, then there exist $C=C(k,R_0,\mu)>0$ and $\e_0 >0$ such that
\begin{align}
	T_{\max}(\eps f) \le 
	\left\{
	\begin{aligned}
	&C \e^{-\frac{2}{d-k}} && k<d, \\
	&\exp (C/\eps) && k=d
	\end{aligned}
	\right.
	\label{thm:ii1}
\end{align}
holds for any $\e \in (0,\e_0)$. 
\end{theorem}

In the case $k=d$, as a main result, we give a refinement of the upper bound of the lifespan.

\begin{theorem}[Refinement of the lifespan] \label{thm:m1}
Let $d\ge1$ and $\e >0$. 
Suppose that $\{g_n\}_n \in \ell^1(\Z)$ satisfies $g_0=1$ and
$\mu := g_0- \sum_{n\neq0} |g_n|>0$.
If $f \in L^{1}_{\mathrm{loc}}(\R^d)$ satisfies 
\begin{align}
	-\Im f(x) \ge 
	\left\{
	\begin{aligned}
	&|x|^{-d}(\log |x|)^{-\a} && |x| >R_0, \\
	&0 && |x| \le R_0,
	\end{aligned}
	\right.
	\label{assmp:1}
\end{align}
for some $\a \in \R$ and $R_0>0$,
then there exist $C=C(\a,R_0,\mu)>0$ and $\e_0 >0$ such that
\begin{align}
	T_{\max}(\eps f) \le 
	\begin{cases}
	\exp(C \e^{-\frac{2}{d}}) \quad \text{if}\; \a >1, \\
	\exp \(C \( \e \log \e^{-1} \)^{-\frac{2}d} \) \quad \text{if}\; \a =1, \\
	\exp(C \e^{-\frac{2}{d+2(1-\a)}}) \quad \text{if}\; \a <1
	\end{cases}
	\label{thm:m2}
\end{align}
holds for any $\e \in (0,\e_0)$. 
\end{theorem}

\begin{remark}
When $\a = 0$, the estimate \eqref{thm:m2} is a refinement of \eqref{thm:ii1}. In fact, we obtain the additional blowup rate $\frac{d}{d+2}$, compared with \eqref{thm:ii1}. 
\end{remark}

\begin{remark}
In addition to the assumption of the theorem, let us suppose that $F(e^{i\theta})$ is Lipschitz continuous and $f\in L^2(\R^d)$.
Then, as mentioned in \cite{MM2}, a standard contraction argument yields a unique solution $u(t) \in C_t(I;L^2_x(\R^d)) \cap L^{\frac{2(d+2)}{d}}_{t,\mathrm{loc}}(I; L_x^{\frac{2(d+2)}{d}}(\R^d))$ in the sense of the integral equation corresponding to \eqref{nls}
\[
	u(t) = U(t) \e f - i \int_{0}^{t} U(t -s) F(u(s)) ds
\]
in $L^2(\R^d)$ for any $t \in I$, where  $I = (-\ta, T)$ for some $\ta$, $T >0$ and $U(t)=e^{it\Delta}$ is the free Schr\"odinger group.
The solution is a weak solution on $[0, T)$ by a suitable extension of $u$ in $(-\infty,-\tau/2)\times\R^d$.
Thus, denoting $I_{\max}$ by a maximal existence interval of the solution, 
$\wt{T} = \sup I_{\max}$ coincides with $T_{\rm{max}}(\e f)$ in Theorem \ref{thm:m1} and
$u(t)$ blows up at $t= \wt{T}$ in such a sense that
$\lim_{t \to \wt{T}-0} \Lebn{u(t)}{2} = \infty$.

\end{remark}

\section{Proof of Theorem \ref{thm:m1}}
We follow the test function method argument as in \cite{IS1}. 
Set $\eta \in C^{\infty}([0,\infty))$ satisfying
\begin{align*}
	\eta(s) 
	\begin{cases}
	= 1 \quad \text{if}\; s \in [0, 1/2] \\
	\text{is decreasing}  \quad \text{if}\; s \in (1/2, 1) \\
	= 0 \quad \text{if}\; s \in [1, \infty),
	\end{cases}
	\quad 
	\eta^{\ast}(s)=
	\begin{cases}
	0 \quad \text{if}\; s \in [0, 1/2) \\
	\eta(s) \quad \text{if}\; s \in [1/2, \infty).
	\end{cases}
\end{align*}
We further define the cut-off functions
\begin{align*}
	\psi_R(t,x) = \left[ \eta\(\frac{|x|^2+t}{R}\)\right]^{2p'_0}, \quad \psi^{\ast}_{R}(t,x) = \left[ \eta^{\ast}\(\frac{|x|^2+t}{R}\)\right]^{2p'_0}
\end{align*}
for any $R>0$ and all $(t,x) \in [0,\infty) \times \R^N$, where $p_0 = 1+\frac2{d}$.
Let us here show the following estimate for the initial data:
\begin{lemma} \label{lem:1}
Let $f \in L^{1}_{\mathrm{loc}}(\R^d)$ satisfy \eqref{assmp:1} for some $\a \in \R$ and $R_0>0$. 
Then there exist constants $C=C(d, R_0, \a)>0$ and $R_1 = R_1(d,R_0, \a)>0$ such that
\begin{align}
	- \int_{\R^d}\Im f(x) \psi_R(0, x) dx \ge
	\left\{
	\begin{aligned}
	&{}C \quad \text{if}\; \a >1, \\
	&{}C \log \log R \quad \text{if}\; \a =1, \\
	&{}C (\log R)^{1-\a} \quad \text{if}\; \a <1
	\end{aligned}
	\right. \label{thm:i1}
\end{align}
holds as long as $R_1 < R$.
\end{lemma}

\begin{proof}[Proof of Lemma \ref{lem:1}]

We may assume $2R_0^2 < R$. A direct computation shows  
\begin{align*}
	&{}- \int_{\R^d}\Im f(x) \psi_R(0, x) dx \\
	\ge{}& \int_{\{|x| > R_0 \}} |x|^{-d}(\log |x|)^{-\a} \left[ \eta\(\frac{|x|^2}{R}\)\right]^{2p'_0}\, dx \\
	\ge{}& \int_{\{R_0 < |x| < \sqrt{R/2} \}} |x|^{-d}(\log |x|)^{-\a} \left[ \eta\(\frac{|x|^2}{R}\)\right]^{2p'_0}\, dx \\
	={}& C \int_{R_0}^{\sqrt{R/2}} r^{-1} (\log r)^{-\a}\, dr
\end{align*}
for any $R > 2R_0^2$. Further we see that
\begin{align}
	\int_{R_0}^{\sqrt{R/2}} r^{-1} (\log r)^{-\a}\, dr \ge
	\left\{
	\begin{aligned}
	&{}C \quad \text{if}\; \a >1, \\
	&{}C \log \log R \quad \text{if}\; \a =1, \\
	&{}C (\log R)^{1-\a} \quad \text{if}\; \a <1
	\end{aligned}
	\right. \label{lem:11}
\end{align}
when $R$ is large enough. Indeed, when $\a =1$, ones easily has
\begin{align*}
	\int_{R_0}^{\sqrt{R/2}} r^{-1} (\log r)^{-\a}\, dr ={}& \log\(\log R -\log 2\) - \log 2 -\log \log R_0. 
\end{align*}
Then it holds that
\begin{align*}
	\log R -\log 2 \ge \frac12 \log R, \quad \log \log R -2\log 2 -\log R_0 \ge \frac12 \log \log R
\end{align*}
as long as $R$ is sufficiently large. Thus, we have the estimate \eqref{lem:11}. The other cases are similar way, taking into account
\begin{align*}
	&{}\frac12 \( \log R - \log 2\) \ge 2\log R_0 \quad \text{if}\; \a<1, \\
	&{}\frac12 \log R - \log 2 \ge 2\log R_0 \quad \text{if}\; \a>1
\end{align*}
as long as $R$ is sufficiently large.
In conclusion, taking $R_1 >0$ satisfying $R_1 > 2R_0^2$ and \eqref{lem:11}, we have \eqref{thm:i1} if $R > R_1$.
\end{proof}

Let us go back to the proof of Theorem \ref{thm:m1}.
Put $T_{\e} = T_{\rm{max}}(\e f)$.
let $u(t,x)$ be a weak solution on $[0,T_{\e})$ with initial condition $u(0)=\eps f$. 
For each $\e >0$, it is deduced that $T_{\e} \le R_1$ or $T_{\e} > R_1$, where $R_1$ is given by Lemma \ref{lem:1}. Since the former case is trivial, we may suppose that $T_{\e} > R_1$.
Taking the real part of the both side and $\psi = \psi_R$ in \eqref{weak:1}, from density argument and the integration by part, we deduce that weak solutions to \eqref{nls} satisfies
\begin{align} 
\begin{aligned}
	&{}\Re \iint_{(0,T_{\e}) \times \R^d}  u(t,x) \{ -i\pa_t (\psi_R(t,x)) + \D(\psi_R(t,x)) \} dt dx \\
	= {}& -\e \int_{\R^d} \Im f(x) \psi_R(0, x) dx \\
	&{}+ \Re \sum_{n \in \Z} g_n\iint_{(0,T_{\e}) \times \R^d} F_n(u(t,x)) \psi_R(t,x) dtdx 
\end{aligned}
	\label{eq:1}
\end{align}
for any $R \in [R_0, T_{\e})$, where $F_n (u) = |u|^{p_0-n}u^n$. 
Let us here introduce
\begin{align*}
	I_{n}(R) ={}& \iint_{(0,T_{\e}) \times \R^d} F_n(u(t,x)) \psi_R(t,x)\, dx dt.
\end{align*}
Then since $|I_n(R)| \le I_{0}(R)$, we have
\begin{align*}
	&{}\Re \sum_{n \in \Z} g_n\iint_{(0,T_{\e}) \times \R^d} F_n(u(t,x)) \psi_R(t,x) dtdx \\
	\ge{}& I_0(R) - \sum_{n \neq 0} |g_n| I_0(R) = (1-\mu) I_0(R).
\end{align*}
Hence, we see from the above and \eqref{eq:1} that
\begin{align*}
	&{}-\e \int_{\R^d} \Im f(x) \psi_R(0, x) dx + (1-\mu) I_0(R) \\
	\le{}& \iint_{(0,T_{\e}) \times \R^d}  |u(t,x)| \( |\pa_t (\psi_R(t,x))| + |\D(\psi_R(t,x))| \) dt dx \\
	\le{}& \iint_{(0,T_{\e}) \times \R^d}  |u(t,x)| \( \frac{C_1}{R} + C_2 \frac{|x|^2}{R^2} \)\(\psi_{R}^{\ast}(t,x)\)^{\frac{1}{p_0}} dt dx.
\end{align*}
Taking $A = C_1 + C_2/R_0$, combining the H\"older inequality with the definition of $p_0$, it is obtained that 
\begin{align*}
	&{}-\e \int_{\R^d} \Im f(x) \psi_R(0, x) dx + (1-\mu) I_0(R) \\
	\le{}& A \( \iint_{(0,T_{\e}) \times \R^d} |u(t,x)|^{p_0}\psi_{R}^{\ast}(t,x)\, dxdt\)^{\frac1{p_0}} \( \int_{(0,R) \times B(0, \sqrt{R})} \frac{1}{R^{p'_0}} dtdx \)^{\frac{1}{p'_0}} \\
	\le{}& C\( \iint_{(0,T_{\e}) \times \R^d} |u(t,x)|^{p_0}\psi_{R}^{\ast}(t,x)\, dxdt\)^{\frac1{p_0}},
\end{align*}
which implies 
\begin{align}
	\begin{aligned}
	&{}-\e \int_{\R^d} \Im f(x) \psi_R(0, x) dx + (1-\mu) I_0(R) \\
	\le{}& C\( \iint_{(0,T_{\e}) \times \R^d} |u(t,x)|^{p_0}\psi_{R}^{\ast}(t,x)\, dxdt\)^{\frac1{p_0}}. \label{eq:2}
	\end{aligned}
\end{align}
Following the idea in \cite[Lemma 3.10]{IS1}, we here define two functions by
\begin{align*}
	y(r) = \iint_{(0,T_{\e}) \times \R^d} |u(t,x)|^{p_0}\psi_{r}^{\ast}(t,x)\, dxdt, \quad Y(R) = \int_0^R y(r) \frac{dr}{r}.
\end{align*}
Note that $y \in C([0, T_{\e}))$ and $Y \in C^1([0,T_{\e}))$ satisfying 
\[
	RY'(R) = \iint_{(0,T_{\e}) \times \R^d} |u(t,x)|^{p_0}\psi_{R}^{\ast}(t,x)\, dxdt.
\]
A direct calculation shows
\begin{align*}
	Y(R) ={}& \int_0^R \( \iint_{(0,T_{\e}) \times \R^d} |u(t,x)|^{p_0}\psi_{r}^{\ast}(t,x)\, dxdt \) \frac{dr}{r} \\
	={}& \iint_{(0,T_{\e}) \times \R^d}  |u(t,x)|^{p_0} \(\int_0^R  \left[ \eta^{\ast}\(\frac{|x|^2+t}{r}\)\right]^{2p'_0} \, \frac{dr}{r}\) dxdt \\
	={}& \iint_{(0,T_{\e}) \times \R^d}  |u(t,x)|^{p_0} \( \int_{\( |x|^2 + t\)/R}^{\infty}  \left[ \eta^{\ast}\(s\)\right]^{2p'_0}\, \frac{ds}{s}\) dxdt.
\end{align*}
Noting that 
\begin{align*}
	\int_{\s}^{\infty} \left[ \eta^{\ast}\(s\)\right]^{2p'_0}\, \frac{ds}{s} \le \left[ \eta\( \s \)\right]^{2p'_0} \log 2
\end{align*}
for any $\s \ge 0$, one sees that
\begin{align*}
	Y(R) \le{}& \log 2 \iint_{(0,T_{\e}) \times \R^d}  |u(t,x)|^{p_0} \left[ \eta \( \frac{|x|^2 +t}{R} \)\right]^{2p'_0}\, dxdt \\
	={}& \log 2 \iint_{(0,T_{\e}) \times \R^d}  |u(t,x)|^{p_0} \psi_R (t,x)\, dxdt \\
	={}& \log 2\ I_0(R).
\end{align*}
Therefore, plugging these estimates and \eqref{thm:i1} into \eqref{eq:2}, it holds that
\begin{align*}
	CRY'(R) \ge{}& \left\{
	\begin{aligned}	
	&{}\( C \e + \frac{(1-\mu)}{\log 2} Y(R) \)^{p_0} \quad \text{if}\; \a >1, \\
	&{}\( C \e \log \log R + \frac{(1-\mu)}{\log 2} Y(R) \)^{p_0} \quad \text{if}\; \a =1, \\
	&{}\( C \e (\log R)^{1-\a} + \frac{(1-\mu)}{\log 2} Y(R) \)^{p_0}  \quad \text{if}\; \a <1
	\end{aligned}
	\right.
\end{align*}
for any $R \in \( R_1, T_{\e} \)$.
Hence we reach to
\begin{align*}
	-\frac{C}{R} \ge
	{}& \left\{
	\begin{aligned}	
	&{}\frac{d}{dR} \( \e + Y(R) \)^{1- p_0} \quad \text{if}\; \a >1, \\
	&{}\frac{d}{dR} \( \e \log \log R + Y(R) \)^{1-p_0} \quad \text{if}\; \a =1, \\
	&{}\frac{d}{dR} \( \e (\log R)^{1-\a} + Y(R) \)^{1-p_0}  \quad \text{if}\; \a <1.
	\end{aligned}
	\right.
\end{align*}
Let us handle the case $\a =1$. 
Integrating the both side in $(R^{1/2}, R)$, it holds that 
\begin{align*}
	\( \e \log \log R + Y(R) \)^{1-p_0} - \( \e \log \log R^{1/2} + Y(R^{1/2}) \)^{1-p_0} \le - C \log R^{1/2},
\end{align*}
which yields
\[
C\log R^{1/2} \le 
\big(
	\varepsilon \log\log R^{1/2}+Y(R^{1/2})
\big)^{1-p_0}\le
\big(
	\varepsilon \log\log R^{1/2}
\big)^{1-p_0}.
\]
Putting $S=\log R^{1/2}$, we can rewrite the above inequality as
\[
\e^{p_0-1}S(\log S)^{p_0-1}\le C^{-1}.
\]
Set $\e_{\ast} =\e^{p_0-1}$ and 
$
S_{\ast}=c_{\ast} \e_{\ast}^{-1}(\log \e_{\ast}^{-1})^{1-p_0}
$
for any $\e \in (0,1)$, where $c_{\ast}$ will be chosen later. 
Then it follows that
\begin{align*}
&{}\e^{p_0-1}S_{\ast}(\log S_{\ast})^{p_0-1} \\
={}& \e_{\ast} \Big(c_{\ast} \e_{\ast}^{-1}(\log \e_{\ast}^{-1})^{1-p_0}\Big)
\Big(\log c_{\ast}+\log \e_{\ast}^{-1}-(p_0-1)\log(\log \e_{\ast}^{-1})\Big)^{p_0-1}
\\
={}&
c_{\ast} \left(1 + \frac{\log c_{\ast} -(p_0-1) \log(\log \e_{\ast}^{-1})}{\log \e_{\ast}^{-1}}\right)^{p_0-1}
\ge c_{\ast} \delta_{\ast}
\end{align*}
for some $\delta_{\ast}>0$. Therefore taking $c_{\ast}=(\delta_{\ast}C)^{-1}$, 
we have
\[
	\e^{p_0-1}S(\log S)^{p_0-1}\le C^{-1}
	\le \e^{p_0-1}S_{\ast}(\log S_{\ast})^{p_0-1}.
\]
This provides that $S \le S_{\ast}$ and therefore 
\[
	R \le \exp(2S_{\ast}) = \exp\( 2c_{\ast} \e_{\ast}^{-1}(\log \e_{\ast}^{-1})^{1-p_0}\)
\]
holds. Taking $R \to T_{\e}$, we conclude that
\begin{align}
	T_{\e} \le \exp \(C \( \e \log \e^{-1} \)^{-\frac{2}d} \). \label{thm:m3}
\end{align}
One also chooses $\eps_0 \in (0,1)$ satisfying $\exp \(C \( \e_0 \log \e_0^{-1} \)^{-\frac{2}d} \) = R_1$. Then, \eqref{thm:m3} is true for all $\eps \in (0,\eps_0)$.
The other cases are easier. Indeed, we establish
\begin{align*}
	&{}\log R \le C \e^{\frac{1-p_0}{1-(1-\a)(1-p_0)}} \quad \text{if}\; \a<1, \\
	&{}\log R \le C \e^{1-p_0} \quad \text{if}\; \a >1,
\end{align*}
which implies the desired estimate \eqref{thm:m2}. The proof is completed.

\subsection*{Acknowledgments} 
The first author was supported by JSPS KAKENHI Grant Numbers 19K14580.
The second author was supported by JSPS KAKENHI Grant Numbers 18K13445.


\end{document}